\documentclass{amsart}
\usepackage[all,pdf]{xy}
\usepackage{amsfonts,amssymb,amsmath,enumerate,verbatim,mathtools,tikz,bm,mathrsfs,tikz-cd,hyperref,comment}
\hypersetup{
    colorlinks=true,
    linkcolor=blue,
    filecolor=magenta,      
    urlcolor=cyan,
}

\usepackage[utf8]{inputenc}


\title[A partial converse ghost lemma]{A partial converse ghost lemma for the  derived category of a commutative noetherian ring}

\author[Jian Liu]{Jian Liu}
\address{School of Mathematical Sciences, University of Science and Technology of China, Hefei 230026, Anhui, P.R. China.}
\email{liuj231@mail.ustc.edu.cn}

\author[Josh Pollitz]{Josh Pollitz}
\address{Department of Mathematics,
University of Utah, Salt Lake City, UT 84112, U.S.A.}
\email{pollitz@math.utah.edu}

\DeclareMathOperator{\h}{H}

\newcommand{\Z}{\mathbb{Z}}

\newcommand{\T}{\mathsf{T}}
\newcommand{\C}{\mathsf{C}}

\newcommand{\D}{\mathsf{D}}

\pgfdeclarelayer{bg}    
\pgfsetlayers{bg,main}

\newcommand{\del}{\partial}

\newcommand{\m}{\mathfrak{m}}

\newcommand{\p}{\mathfrak{p}}

\newcommand{\e}{\epsilon}

\DeclareMathOperator{\Spec}{Spec}

\DeclareMathOperator{\Hom}{Hom}

\DeclareMathOperator{\thick}{\mathsf{thick}}

\newcommand{\shift}{{\mathsf{\Sigma}}}
\newcommand{\gin}{{\mathsf{gin}}}
\newcommand{\cogin}{{\mathsf{cogin}}}
\newcommand{\level}{{\mathsf{level}}}

\newcommand{\xra}{\xrightarrow}

\newtheorem{theorem}{Theorem}[section]

\newtheorem{proposition}[theorem]{Proposition}

\newtheorem{lemma}[theorem]{Lemma}
\newtheorem{corollary}[theorem]{Corollary}

\newcounter{intro}
\newtheorem{introthm}[intro]{Theorem}

\theoremstyle{definition}
\newtheorem{example}[theorem]{Example}
\newtheorem{remark}[theorem]{Remark}

\newtheorem{chunk}[theorem]{}

\newtheorem*{ack}{Acknowledgements}

\begin{document}

\maketitle

\begin{abstract}
In this article a condition is given to detect the containment among thick subcategories of the bounded derived category of a commutative noetherian ring. More precisely, for a   commutative noetherian ring $R$
 and    complexes of $R$-modules with finitely generated  homology $M$ and $N$, we show  $N$ is in the thick subcategory generated by $M$ if and only if the ghost index of $N_\p$ with respect to $M_\p$ is finite for each prime $\p$ of $R$. To do so, we establish a ``converse coghost lemma"  for the bounded derived category of a non-negatively graded DG algebra with noetherian homology.
\end{abstract}
\section*{Introduction}

This article is concerned with certain numerical invariants  and thick subcategories  in the bounded derived category of a  commutative noetherian ring. Let $R$ be a commutative noetherian ring,  $\D(R)$ will  denote its derived category, and  $\D^f_b(R)$ will be the full subcategory of $\D(R)$ consisting of objects with finitely generated  homology. 

An object $N$ of  $\D(R)$ is in the thick subcategory generated by  $M$, denoted $\thick_{\D(R)}(M)$, provided that $N$  can be inductively built from $M$ using the triangulated structure of $\D(R)$ (see \ref{level} for more details). 
 There are  cases where a notion of support reports on whether $N$ is in $\thick_{\D(R)}(M)$. For example, there is the celebrated  theorem of Hopkins \cite[Theorem 11]{H} and Neeman \cite[Theorem 1.5]{N} that applies when $M$ and $N$ are perfect complexes.
Another instance is when $R$ is locally complete intersection by using support varieties;  this was proved by Stevenson for thick subcategories containing $R$   when $R$ is a quotient of a regular ring \cite[Corollary 10.5]{Stev}, and in general in \cite[Theorem 3.1]{JJ}. 
However, in general, detecting containment among thick subcategories  can be an intractable task.
 
In this article, we give a new criterion to determine the containment among thick subcategories of $\D^f_b(R)$ based on certain numerical invariants being locally finite.   We quickly define these below; see \ref{level}, \ref{coghost}, and \ref{ghostchunk} for precise definitions.

For a triangulated category $\T$, fix objects $G$ and $X$. The level of $X$ with respect to $G$ counts the minimal number of cones needed to generate $X$, up to suspensions and direct summand, starting from $G$. We  denote this by $\level_{\T}^G(X)$ and note that this is finite exactly when $X$ is in $\thick_{\T} (G)$.  The coghost index of $X$ with respect to $G$, denoted $\cogin_\T^G (X)$, is  the minimal number  $n$ satisfying that any composition 
\[
X^n\xra{f^n}X^{n-1}\to \ldots \xra{f^1} X^0=X,
\] where each $\Hom_\T(f^i,\shift^j G)=0$, must be zero in $\T$. Switching the variance in the definition above determines the ghost index of $X$  with respect to $G$, denoted $\gin_\T^G(X). $

These invariants are of independent interest and have been studied in \cite{ABI,ABIM,Beli,Bergh,BonBer, Christensen,Kelly,Letz,Rouq}. In general,  they  are related by the following well-known  \emph{(co)ghost lemma}:
\[
\max\{\cogin_\T^G(X), \gin^G_\T (X)\}\leq \level_\T^G(X).\] 
Oppermann and \v{S}\'{t}ov\'{i}\v{c}ek proved a so-called  \emph{converse  coghost lemma:} Namely, 
$
\cogin_{\D^f_b(R)}^M(N)$ and $\level_{\D^f_b(R)}^M (N)$ agree  whenever $M$ and $N$ are objects of $\D^f_b(R)$, see \cite[Theorem 24]{OS}. 
In general, it remains open whether the analogous equality holds for objects of $\D^f_b(R)$ when $\cogin$ is replaced by $\gin$; furthermore,  \cite[2.11]{Letz} notes that  the techniques used in \cite[Theorem 24]{OS} cannot be suitably adapted to prove this.

 In this article we ask whether finiteness of certain ghost  indices can determine finiteness of $\level$, and hence  containment among thick subcategories. 
The main result in this direction is the following which is contained in Theorem \ref{main theorem}.
\begin{introthm}\label{thmintro}
Let $R$ be a commutative noetherian ring. For $M$ and $N$ in $\D^f(R),$ the following are equivalent: 
\begin{enumerate}
    \item  \label{e2}$\level^M_{\D^f_b(R)} (N)<\infty$;
    \item \label{e3}$\gin_{\D^f_b(R_\p)}^{M_\p} (N_\p)<\infty$ for all prime ideals $\p$ of $R$. 
\end{enumerate}
\end{introthm}
One of the main steps in the proof of Theorem \ref{thmintro} is establishing  a converse coghost lemma  for graded-commutative, bounded below DG algebras $A$ with  $\h(A)$ a noetherian $\h_0(A)$-module (cf. Theorem \ref{t:converse}).
We follow the proof of  \cite[Theorem 24]{OS} closely, however,  extra care is needed when working with such DG algebras.  Namely,  we make use of certain ascending semifree filtrations, see \ref{filtration}, as the truncations used by  Oppermann and \v{S}\'{t}ov\'{i}\v{c}ek are no longer available in this setting.

\begin{ack}
We thank Srikanth Iyengar and Janina Letz for helpful discussions.  The second author was supported by NSF grant DMS 2002173.
\end{ack}

\section{Derived Category of a DG Algebra and Semifree DG Modules}

Much of this section is devoted to setting notation and reviewing the necessary background regarding the topics from the title of the section. Proposition \ref{truncation} is the main technical result of the section and will be put to use in the next section. 

Throughout this article objects will be graded homologically. By a DG algebra we will implicitly assume $A$ is  non-negatively graded and  graded-commutative. For the rest of the section fix a DG algebra $A$.

\begin{chunk}
Let $\D(A)$ denote the derived category of (left) DG $A$-modules (see,   for example, \cite[Sections 2 \& 3]{ABIM} or  \cite[Section 4]{Keller}).  We use  $\shift$ to denote the suspension functor on the triangulated category $\D(A)$ where  $\shift$ is the autoequivalence defined   by  \[(\shift M)_i\coloneqq M_{i-1},~ a\cdot(\shift m)\coloneqq (-1)^{\mid a\mid}am,\text{ and }\del^{\shift M}\coloneqq-\del^M.\] 
For a DG $A$-module $X,$ its homology  $\h(M)=\{\h_i(M)\}_{i\in \mathbb{Z}}$ is naturally a graded $\h(A)$-module.  Also, define the \emph{infimum of X} to be  $\inf(M)\coloneqq\inf \{n\mid \h_n(X)\neq 0\};$ its \emph{supremum} is  $\sup (M)\coloneqq \sup\{n\mid \h_n(M)\neq 0\}.$
\end{chunk}
\begin{chunk}
The following triangulated subcategories of $\D(A)$ will be of interest in the sequel. First,  
 let $\D^f(A)$ denote the full subcategory of $\D(A)$ consisting of DG $A$-modules $M$  such that each $\h_i(M)$ is a  noetherian    $\h_0(A)$-module. We let $\D^f_+(A)$  be the full subcategory of objects $M$ of  $\D^f(A)$ such that $\inf(M)>-\infty$. Finally, $\D^f_b(A)$ consists of those objects $M$  of $\D^f(A)$  satisfying $\h_i(M)=0$ for all $|i|\gg 0$. When $\h(A)$ is noetherian  as a module over $\h_0(A)$ and $\h_0(A)$ is noetherian, $\D^f_b(A)$ is exactly the full subcategory of $\D(A)$ whose objects $M$ are those with $\h(M)$ being finitely generated as a graded $\h(A)$-module. 
\end{chunk}

\begin{chunk}\label{semifree resolution}\label{triangle}
A DG $A$-module $F$ is  \emph{semifree} if it admits a filtration of DG $A$-submodules 
\[
\ldots \subseteq F(-1)\subseteq F(0)\subseteq F(1)\subseteq \ldots 
\] where  $F(i)=0$ for $i\ll 0$, $F=\cup F(i)$ and each 
$F(i)/F(i-1)$ is a direct sum of shifts of $A$.  The filtration above is called a \emph{semifree filtration} of $F$.\footnote{The choice to allow arbitrary indices for the start of the filtration is a non-standard one but  simplifies notation in the proof of Theorem \ref{t:converse}.}  
By \cite[Section 3]{Keller}, $F$ is homotopy colimit of the $F(i)$ and so  there is the following exact triangle in $\D(A)$
\begin{equation}\label{exact triangle}
   \coprod_{i\in \mathbb{Z}}F(i)\xra{1-s} \coprod_{i\in \mathbb{Z} }F(i)\rightarrow F\rightarrow \shift \coprod_{i\in \mathbb{Z}}F(i),
\end{equation}
 where $s$ is induced by the canonical inclusions  $F(j)\hookrightarrow F(j+1)\hookrightarrow \coprod F(i)$.
\end{chunk}

\begin{chunk} For the following background on semifree resolutions see \cite[Chapter 6]{FHT} (or \cite[Section 1.3]{IFR}). 
Let $M$ be a DG $A$-module.  There exists a surjective quasi-isomorphism of DG $A$-modules $\epsilon\colon F\xra\simeq M$ where $F$ is a semifree DG $A$-module, see \cite[Proposition 6.6]{FHT}; the map $\epsilon$  is called a \emph{semifree resolution} of $M$ over $A$. Semifree resolutions of $M$ are unique up to homotopy equivalence. 
\end{chunk}

\begin{chunk}\label{hominD(A)chunk}
Fix a DG $A$-module $M$ with semifree resolution $\e\colon F\xra{\simeq} M.$ For any DG $A$-module $N$, it is clear that
\[
\Hom_{\D(A)}(M,N)=\Hom_{\D(A)}(F,N)
\] and the right-hand side is computed as the degree zero homology of the DG $A$-module $\Hom_A(F,N).$ That is, 
\begin{equation}\label{hominD(A)}
\Hom_{\D(A)}(M,-)=\h_0(\Hom_A(F,-)).
\end{equation}In particular, $\Hom_{\D(A)}(M,N)$ naturally inherits an  $\h_0(A)$-module structure and since $A$ is non-negatively graded, $\Hom_{\D(A)}(M,N)$ inherits an $A_0$-module structure. As semifree resolutions are homotopy equivalent, this $\h_0(A)$-module is independent of choice of semifree resolution. 
\end{chunk}

\begin{chunk}\label{filtration} Assume  each $\h_i(A)$ is finitely generated over $\h_0(A)$ and $\h_0(A)$ is itself noetherian. 
Let $M$ be an object of $\D^f_+(A)$. By \cite[Appendix B.2]{AINSW}, there exists a semifree resolution $F\xra\simeq M$ with $F_i=0$ for all $i<\inf(M)$ and $F$ admits a semifree filtration $\{F(i)\}_{i\in \mathbb{Z}}$ equipped with exact sequences of DG $A$-modules  
\[
0\to F(i-1)\to F(i)\to \shift^{i} A^{\beta_i}\to 0
\] for some non-negative integer $\beta_i\geq 0$.
\end{chunk}

\begin{lemma}\label{null Lemma} Assume $\h_0(A)$ is  noetherian  and that each $\h_i(A)$ is finitely generated over it.
    Let $N$ be an object of $\D(A)$ such that $\sup(N)<\infty$. For an object $M$ in $\D(A)$ with $\inf(M)> \sup(N)$,
$    \Hom_{\D(A)}(M,N)=0.$
\end{lemma}
\begin{proof}
Fix a semifree resolution  $F\xra\simeq M$ as in \ref{filtration}. 
By (\ref{hominD(A)}) in \ref{hominD(A)chunk}, 
\[  \Hom_{\D(A)}(\shift^{i}A^{\beta_i},N)\cong \h_{i}(N)^{\beta_i}=0\]
for each $i\geq \inf(M).$
Combining these isomorphisms with the  exact sequences 
 \[ 0\rightarrow F(i-1)\rightarrow F(i)\rightarrow \shift^{i}A^{\beta_i}\rightarrow 0\]
yields by induction that 
 $\Hom_{\D(A)}(F(i),N)=0$ for all $i\geq \inf(M)$. Finally, (\ref{exact triangle}) in \ref{triangle}
implies $\Hom_{\D(A)}(F,N)=0$, and hence  $\Hom_{\D(A)}(M,N)=0$ (cf.  \ref{hominD(A)chunk}). 
\end{proof}

\begin{proposition}\label{truncation}
Assume $\h_0(A)$ is noetherian  and each $\h_i(A)$ is a finitely generated  $\h_0(A)$-module.
Let  $M$ be in $\D^f_+(A)$ and $N$ be an object in $\D(A)$ such that $\sup(N)<\infty$. Suppose $F\xra{\simeq} M$ is a semifree resolution of $M$ as in \ref{filtration}, then for all $i> \sup(N)$ the natural map below is an isomorphism
\[
\Hom_{\D(A)}(M,N)\xra{\cong} \Hom_{\D(A)}(F(i),N).
\]
\end{proposition}
\begin{proof}
For each $i\geq \inf(M)$, there is an exact sequence of DG $A$-modules 
\begin{equation}\label{exact sequence}
0\rightarrow F(i)\rightarrow F\rightarrow F'\rightarrow 0
\end{equation}
 where by choice of $F$ we have that 
$\inf(F')>  i$. Applying $\Hom_{\D(A)}(-,N)$ to (\ref{exact sequence}) and appealing to  Lemma  \ref{null Lemma} yields the desired isomorphisms whenever  $i> \sup (N).$
\end{proof}

\section{Levels and Coghost Index in $\D(A)$}
\label{sectioncoghost}
We begin by briefly recalling the notion of level. For more details,  see  \cite[Section 2]{ABIM}, \cite[Section 2]{BonBer} or \cite[Section 3]{Rouq}.
\begin{chunk}\label{construction}\label{level}
Let $\T$ be a triangulated category and $\C$ be a full subcategory of $\T$. We say $\C$ is  \emph{thick} if it is closed under suspensions, retracts and cones. 
The smallest thick subcategory of $\T$ containing an object $X$ is denoted $\thick_\T (X);$ this consists of all objects  $Y$ such that one can obtain $Y$ from $X$ using finitely many suspensions, retracts and cones. 

We set $\level_\T^X (Y)$ to be smallest non-negative integer $n$ such that  $Y$ can be built starting from $X$ using finitely many suspensions, finitely many retracts and exactly $n-1$ cones in $\T.$  If no such $n$ exists, we set $\level_\T^X(Y)=\infty$. 
Note $Y$ is in $\thick_\T (X)$ if and only if $\level_\T^X (Y)<\infty$. Also,  if $\C$ is a thick subcategory of $\T$ containing $X$, then
$
\level^X_\T (Y)=\level^X_\C (Y).$
\end{chunk}

\begin{example}\label{example}
Let $A$ be a DG algbera. A DG $A$-module $M$ is \emph{perfect} if $M$ is an object of $\thick_{\D(A)}(A).$ In this case, $M$ is a retract of  a semifree DG $A$-module $F$ with finite semifree filtration
\[
0\subseteq F(0)\subseteq F(1)\subseteq \ldots \subseteq F(n)=F.
\] If $n$ is the minimal such value, then $\level_{\D(A)}^A (M)=n-1$, see \cite[Theorem 4.2]{ABIM}.
\end{example}

\begin{chunk}\label{coghost}
 Let $\T$ be a triangulated category with suspension functor $\shift$. A morphism $f\colon X\rightarrow Y$ in $\T$ is called $G$-\emph{coghost} if 
\[
\Hom_\T(f,\shift^iG) \colon \Hom_\T(Y,\shift^iG)\rightarrow \Hom_\T(X,\shift^iG)
\]
is zero for all $i\in \Z$.  
Following \cite[Defnition 2.4]{Letz}, we define the \emph{coghost index of $X$ with respect to $G$ in $\T$}, denoted $\cogin_T^G(X)$, to be the smallest non-negative integer $n$ such that any  composition of $G$-ghost maps 
\[
X^n\xra{f^n}X^{n-1}\xra{f^{n-1}}\ldots\to X^1 \xra{f^1} X^0=X
\] is zero in $\T.$
\end{chunk}
\begin{chunk}
Let $\T$ be a triangulated category with objects $G$ and $X$. 
In this generality, $\level$ bounds  $\cogin$ from above. That is,  
\[ \cogin_{\T}^G(X)\leq \level_\T^G(X),\]
see \cite[Lemma 2.2(1)]{Beli} (see also  \cite[Lemma 4.11]{Rouq}). However, there are known instances when equality holds. For example, $\level^G_\T (-)$ and  $\cogin_\T^G(-)$ agree provided  every object in $\T$ has an appropriate  left approximation by $G$, see   \cite[Lemma 2.2(2)]{Beli} for more details. Another instance   is when  $R$ is a commutative noetherian ring (or more generally, a noether algebra)  
\[
\cogin_{\D^f_b(R)}^G (X)=\level_{\D^f_b(R)}^G(X)
\] for each $G$ and $X$ in $\D^f_b(R);$ this has been coined the \emph{converse coghost lemma} (see \cite[Theorem 24]{OS}). 
\end{chunk}

We now get to the main result of the section which generalizes a particular case of \cite[Theorem 24]{OS} mentioned above. It is worth noting that   \cite[Theorem 24]{OS} was proved for  derived categories satisfying certain finiteness conditions; however, it does not apply directly to the case considered in the theorem below. The proof of \cite[Theorem 24]{OS}  is suitably adapted to the setting under consideration with the main  observation being that  truncations need to be replaced with the ascending filtrations discussed in \ref{filtration}. We have indicated the necessary changes below, while attempting to not recast the parts of the proof of \cite[Theorem 24]{OS} that carry over with only minor changes.

\begin{theorem}\label{t:converse}
 Let $A$ be a DG algebra with  $\h(A)$ noetherian as an $\h_0(A)$-module.  If  $M$ and $N$ are in $\D^f_b(A)$, then 
\[
\cogin_{\D^f_b(A)}^M(N)=\level_{\D(A)}^M(N).
\]
\end{theorem}

\begin{remark}\label{remarksameproof}
For $M$ and $N $ in $\D_b^f(R)$,
\[
\cogin_{\D_+^f(A)}^M(N)=\level_{\D_+^f(A)}^M(N). 
\]
Indeed, one can directly apply the argument from 
\cite[Theorem 24]{OS}  once it is noted that, by restricting scalars along the  map  of commutative rings $A_0\to \h_0(R)$, $\Hom_{\D_+^f(A)}(X,\shift^iN)$ is finitely generated over  $A_0$ for  $X$ in $\D_+^f(A)$ and  $i\in \mathbb{Z}$. 

To see the latter holds, 
 such an $X$ admits a semifree filtration whose subquotients are perfect DG $A$-module. Also, since $N$ is in $\D^f_b(A)$ we can apply Proposition \ref{truncation} to get
\[\Hom_{\D_+^f(A)}(X,\shift^iN)\cong  \Hom_{\D_+^f(A)}(P,\shift^iN)\]
where $P$ a perfect DG $A$-module with a finite semifree filtration as in Example \ref{example}. Therefore, induction on the length of this filtration finishes the proof of the claim, where we are again using that $N$ is in $\D^f_b(A)$.
\end{remark}
Before beginning the proof of Theorem \ref{t:converse}, we record  an  easy but important lemma. 

\begin{lemma}\label{compatible}Let $A$ be a DG algebra.
    Assume $\alpha\colon F^1\to F^2$ is a morphism of bounded below semifree DG $A$-modules with  $F^j_i=0$ for $i<\inf (F^j)$ and semifree filtrations $\{F^j(i)\}_{i\in \mathbb{Z}}$ for $j=1,2$  satisfying
    \[
    0\to F^j(i-1)\to F^j(i)\to \coprod_{\ell\leq i}\shift^\ell A^{\beta^j_\ell(i)}\to 0
    \] for non-negative integers $\beta^j_\ell(i)$ and $j=1,2.$ For each $i\in \mathbb{Z}$, $\alpha$ restricts to a morphism of DG $A$-modules $
    \alpha(i)\colon F^1(i)\to F^2(i).
    $
\end{lemma}
\begin{proof}
Indeed, $F^1(i)=0$ for all $i<\inf(F^1)$ and so there is nothing to show for such values of $i$. Now for $i\geq \inf (F^1)$, the DG $A$-module $F^1(i)$ is generated in degrees at most $i$ and since $\alpha$ is degree preserving $\alpha(F^1(i))$ is generated in degrees at most $i.$ However,  the assumption on the filtration $\{F^2(j)\}$ also implies 
\[
\alpha(F^1(i))\subseteq F^2(i).
\] Hence,  setting $\alpha(i)\coloneqq \alpha\mid_{F^1(i)}$ proves the claim by induction. 
\end{proof}

\begin{proof}[Proof of Theorem \ref{t:converse}]
First, by  \ref{level} and  Remark \ref{remarksameproof}  \[\level_{\D(A)}^M(N)=\level_{\D^f_+(A)}^M(N)=\cogin_{\D_+^f(A)}^M(N),\] while the inequality  \begin{equation}\label{equationtoshow}\cogin_{\D_+^f(A)}^M(N)\geq \cogin_{\D^f_b(A)}^M(N)\end{equation} is standard. So it suffices to prove  the reverse inequality of (\ref{equationtoshow}) holds.

Set $n=\cogin_{\D^f_b(A)}^M(N)$ and consider a composition
\[
N^n\xrightarrow {f^n}N^{n-1}\xrightarrow{f^{n-1}} \ldots\xrightarrow{f^2} N^1\xrightarrow{f^1} N^0=N,
\]
where each $f^i$ is a $M$-coghost map in $\D_+^f(A)$. 
Using the assumptions on $\h(A)$ and that  each $N^i$ is in $\D_+^f(A)$,
there exist semifree resolution $F^i\xra{\simeq} N^i$ with corresponding semifree filtrations $\{F^i(j)\}_{j\in \mathbb{Z}}$ as in \ref{filtration}. Moreover, by \ref{hominD(A)chunk}(\ref{hominD(A)}), each $f^i$ determines a morphism of DG $A$-modules $\alpha^i\colon F^i\to F^{i-1}$ such that the following diagram commutes in $\D(A)$
\begin{equation}\label{diagram}
\begin{tikzcd}
F^i\arrow["\alpha^i"]{r} \arrow["\simeq"]{d} & F^{i-1}\arrow["\simeq"]{d}\\
N^i\arrow["f^i"]{r} & N^{i-1}.
\end{tikzcd}
\end{equation}
Furthermore, by Lemma \ref{compatible} there are the following commutative diagrams of DG $A$-modules
\begin{equation}\label{diagram2}
    \begin{tikzcd}
     F^{i}(j)\arrow[hookrightarrow]{d} \arrow["\alpha^i(j)"]{r} &F^{i-1}(j)\arrow[hookrightarrow]{r} \arrow[hookrightarrow]{dr} &F^{i-1}(j')\arrow[hookrightarrow]{d}\\
    F^i\arrow["\alpha^i"]{rr} & & F^{i-1}
    \end{tikzcd}
\end{equation} whenever $j'\geq j.$  Moreover, since each $F^i(j)$ is a perfect DG $A$-module and  $M$ is in $\D^f_b(A)$,   the commutativity of the diagrams in (\ref{diagram}) and the assumption that each  $f^i$ is $M$-coghost imply     the compositions along the top of (\ref{diagram2}) are $M$-coghost for all  $j'\geq j\gg 0;$ the same argument as in  proof of \cite[Theorem 24]{OS} works in this setting. 

Combining this  with Proposition \ref{truncation} there exists integers $i_j$ such that
\begin{center}
    \begin{tikzcd}
    F^n(i_n)\arrow["\beta^n"]{r} \arrow[hookrightarrow]{d} & F^{n-1}(i_{n-1})\arrow["\beta^{n-1}"]{r} \arrow[hookrightarrow]{d} & \ldots \arrow["\beta^1"]{r} & F^0(i_0)\arrow[hookrightarrow]{d}\\
      F^n\arrow["\alpha^n"]{r} \arrow["\simeq"]{d} & F^{n-1}\arrow["\alpha^{n-1}"]{r} \arrow["\simeq"]{d} & \ldots \arrow["\alpha^1"]{r} & F^0\arrow["\simeq"]{d}\\
    N^n\arrow["f^n"]{r} & N^{n-1}\arrow["f^{n-1}"]{r}  & \ldots \arrow["f^1"]{r} & N^0
    \end{tikzcd}
\end{center} commutes in $\D(A)$, the natural map 
\begin{equation}\label{isohom}
\Hom_{\D_+^f(A)}(F^n,N)\xra{\cong} \Hom_{\D_+^f(A)}(F^n(i_n),N)
\end{equation} is an isomorphism and each $\beta^i$ is $M$-coghost. 
Now since each $\beta^i$ is an $M$-coghost map between perfect DG $A$-modules then by choice of $n$ the composition along the top and then down to $N$, denoted $\beta,$ must be zero. It is worth noting that the previous step needs the assumption that $\h(A)$ is finitely generated over $\h_0(A)$ since, in this case, each map in the  composition defining $\beta$ must be in  $\D^f_b(A).$

Finally, the isomorphism in (\ref{isohom}) identifies $\beta$ with $f=f^1f^2\ldots f^n$. 
Hence,  $f=0$ and so  $\cogin_{\D_+^f(A)}^M(N)\leq n=\cogin_{\D^f_b(A)}^M(N),$ as needed. 
\end{proof}

\begin{chunk}\label{ghostchunk}Let $\T$ be a triangulated category and fix $G$ and $X$ in $\T.$ The 
 \emph{ghost index of $X$ with respect to $G$ in $\T$}, denoted $\gin_T^G (X)$, to be the least non-negative integer $n$ such that any  composition of $G$-ghost maps
\[
X=X^n\xra{f^n}X^{n-1}\xra{f^{n-1}}\ldots\to X^1 \xra{f^1} X^0
\] is zero in $\T;$ this where a map $g$  is $G$-\emph{ghost} provided $\Hom_\T(\shift^i G, g)=0$ for all $i\in \mathbb{Z}$. That is, $\gin_\T^G(X)=\cogin_{\T^{op}}^G(X).$  In general, 
$\gin_{\T}^G(X)\leq \level_\T^G(X)$ and it is unknown whether equality holds when $R$ is a commutative noetherian ring and $\T=\D^f_b(R)$. The point of the next section is to provide  a partial ``converse."
\end{chunk}

\section{A Partial Converse Ghost Lemma}
\label{sectionmain}

In this section $R$ is a commutative noetherian ring. As localization defines  an exact functor $\D(R)\to \D(R_\p)$,  $\level$ cannot increase upon localization. Hence,  for $M$ and $N$ in $\D^f_b(R)$, if $N$ is in $\thick_{\D(R)}(M)$, then \[\gin_{\D^f_b(R_\p)}^{M_\p}(N_\p)<\infty\text{ for all }\p\in \Spec R.\] The converse and an evident corollary are established below.

\begin{theorem}\label{main theorem}
Let $R$ be a commutative noetherian ring and fix $M$ and $N$ in $\D^f_b(R)$. 
If $
\gin_{\D^f_b(R_\p)}^{M_\p}(N_\p) <\infty$ for all $\p\in  \Spec R$, then $N$ is an object of $\thick_{\D(R)}(M).$
\end{theorem}
\begin{corollary}\label{implication}If
$\gin_{\D^f_b(R_\p)}^{M_\p} (N_\p)<\infty$ for all $\p\in \Spec R$, then $ \gin_{\D^f_b(R)}^{M} (N)<\infty.$
\end{corollary}


To prove Theorem \ref{main theorem}, there are essentially two steps. We first go to  derived categories of certain Koszul complexes where it is shown that $\cogin$, $\gin$ and $\level$ all agree using Theorem \ref{t:converse}. Second, we  apply a  local-to-global principle to  conclude the desired result. We explain this below and give the proof of the theorem at the end of the section. 

\begin{chunk}\label{l2g}
 Assume $R$ is local with maximal ideal $\m$, we let $K^R$  be the Koszul complex on a minimal generating set for $\m$. It is regarded as a DG algebra in the usual way and is well-defined up to an isomorphism of DG $R$-algebras, see \cite[Section 1.6]{BH}. 
For any   $\p\in \Spec R$,  let $M$ be an object of $\D(R).$ We set 
\[M(\p)\coloneqq M_\p\otimes_{R_\p} K^{R_\p}\]   which is a DG $K^{R_\p}$-module.
Restricting scalars along the morphism of DG algebras $R_\p \to K^{R_\p}$ we may regard $M(\p)$ as an object of $\D(R_\p)$. 
In  \cite[Theorem 5.10]{BIK2},  Benson, Iyengar and Krause proved the following local-to-global principle: For objects $M$ and $N$ in $\D^f_b(R)$,
$N$
 is in $\thick_{\D(R)}(M)$ if and only if $N(\p)$ is in $\thick_{\D(R_\p)}(M(\p))$. 
\end{chunk}

\begin{lemma}\label{l:gincogin}
Let $R$ be a commutative noetherian local ring. 
For $M$ and $N$ in $\D^f_b(K^R)$, 
\[
\level_{\D(K^R)}^{M}(N)=\cogin_{\D^f_b(K^R)}^M(N)=\gin_{\D^f_b(K^R)}^M(N).
\]	
\end{lemma}
\begin{proof}
The natural map $K^R\to K^{\widehat{R}}$
 is a quasi-isomorphism of DG algebras and so it induces an exact  equivalence 
\[
\D^f_b(K^R)\xra{\equiv} \D^f_b(K^{\widehat{R}}).
\] Since $\cogin$,  $\gin$ and $\level$ are invariant under exact  equivalences we can assume $R$ is complete and set $K=K^R$. 

As  $R$ is complete, it is well known that $R$ admits a dualizing DG module $\omega$; see, for example,  \cite[Corollay 1.4]{Kaw}. 
Now applying \cite[Theorem 2.1]{FIJ}, $\Hom_R(K,\omega)$ is a dualizing DG $K$-module. In particular, setting $(-)^\dagger\coloneqq \Hom_{K}(-,\Hom_R(K,\omega))$ then for any $M$ in $\D^f_b(K)$, $M^\dagger$ is in $\D^f_b(K)$ and the natural biduality map $M\xra{\simeq} M^{\dagger\dagger}$ is an isomoprhism in  $\D^f_b(K).$ Hence,  $(-)^\dagger$  restricts to an exact auto-equivalence of $\D^f_b(K)$. 

Finally, as  $(-)^\dagger$  is an exact auto-equivalence of $\D^f_b(K)$   interchanging coghost and ghost maps, from Theorem \ref{t:converse} the desired equality follows. 
\end{proof}

\begin{remark}The lemma holds for any  DG alegbra $A$ satisfying the hypotheses of Theorem \ref{t:converse} which admits a dualizing DG module as defined in \cite[1.8]{FIJ}. Another
 example, generalizing the Koszul complex above, would be the DG fiber of any local ring map of finite flat dimension whose target ring admits a dualizing complex \cite[Theorem VI]{FIJ}. 
\end{remark}

\begin{lemma}\label{l:koszulgin}
	Let $R$ be a commutative noetherian local ring and let $\mathsf{t}\colon \D(R)\to \D(K^R)$ denote  $-\otimes_R K^R$.
	If  $M$ and $N$ are objects of $\D^f_b(R)$, then 
	\[
	\gin_{\D(K^R)}^{\mathsf{t} M}(\mathsf{t} N )\leq \gin_{\D(R)}^M(N).
	\]
\end{lemma}
\begin{proof} We set $K=K^R$. 
For $X$ in  $\D(R)$ and $Y$ in $\D(K)$, there is an adjunction isomorphism
\begin{equation}\label{adjunction}
  \Hom_{\D(K)}(\mathsf{t} X ,Y)\cong \Hom_{\D(R)}(X,Y),
\end{equation}
which is induced by the natural map   $\eta_X\colon X\to \mathsf{t} X$ given by $x\mapsto x\otimes 1.$
Moreover, when 
 $f\colon Y\rightarrow Z$ is a $\mathsf{t} M$-ghost map in $\D^f_b(K)$, then (\ref{adjunction}) implies that $f$ is a $M$-ghost map in $\D^f_b(R)$. 

Assume $n\coloneqq \gin^M_{\D^f_b(R)}(N)<\infty$ and 
Suppose $g\colon\mathsf{t} N\to Y$ in $\D^f_b(K)$ factors  as the composition of $n$ maps in  $\D^f_b(K)$ which are $\mathsf{t} M$-ghost.
Hence,  by the adjunction above  $g$ is the composition of $n$ maps in $\D^f_b(R)$ which are $M$-ghost, and thus  so is $g\circ \eta_N$. Therefore,  by assumption $g\circ \eta_N=0$ and so from  (\ref{adjunction}) we conclude that $g=0$ in $\D^f_b(K)$, completing the proof.
\end{proof}

\begin{proof}[Proof of Theorem \ref{main theorem}] Let $\p\in \Spec R$ , then by assumption $\gin_{\D^f_b(R_\p)}^{M_\p}(N_\p)<\infty$. Also, 
\begin{align*}
	\gin_{\D^f_b(R_\p)}^{M_\p}(N_\p)&\geq \gin_{\D^f_b(K^{R_\p})}^{M(\p)}(N(\p))\\
	&=\cogin_{\D^f_b(K^{R_\p})}^{M(\p)}(N(\p))\\ 
	&=\level_{\D(K^{R_\p})}^{M(\p)}(N(\p))
\end{align*}
where the inequality is from  Lemma \ref{l:koszulgin}  and the equalities are from Lemma \ref{l:gincogin}. Thus  $\level_{\D(K^{R_\p})}^{M(\p)}(N(\p))<\infty$,  and so  $N(\p)$ is in $\thick_{\D(K^{R_\p})} (M(\p))$ for all $\p\in \Spec R$. Now  by restricting scalars along $R_\p\to K^{R_\p}$ we conclude that $N(\p)$ is in $\thick_{\D(R_\p)} (M(\p))$ for all $\p\in \Spec R$. 
Finally,  we apply \ref{l2g} to obtain the desired result. 
\end{proof}

 \bibliographystyle{amsplain}
\providecommand{\bysame}{\leavevmode\hbox to3em{\hrulefill}\thinspace}
\providecommand{\MR}{\relax\ifhmode\unskip\space\fi MR }
\providecommand{\MRhref}[2]{%
  \href{http://www.ams.org/mathscinet-getitem?mr=#1}{#2}
}
\providecommand{\href}[2]{#2}

\end{document}